\documentclass[12pt,A4]{amsart}
\usepackage{amsmath}
\usepackage{amsthm}
\usepackage{tikz}
\usepackage{tikz-cd}
\usepackage{amssymb}
\usepackage{graphicx}
\usepackage{thmtools}
\usepackage{thm-restate}
\usepackage{setspace}
\usepackage{gensymb}
\usepackage[shortlabels]{enumitem}
\usepackage[hidelinks]{hyperref}
\usepackage{cleveref}
\usepackage{fullpage}

\newtheorem{theorem}{Theorem}[section]
\newtheorem{lemma}{Lemma}[section]
\theoremstyle{definition}
\newtheorem{definition}{Definition}[section]
\newtheorem{example}{Example}[section]
\newtheorem*{remark}{Remark}
\allowdisplaybreaks
\usepackage{comment}
\usepackage{xcolor}

\allowdisplaybreaks[1]
\begin{document}
\title{Wiman--Valiron discs and the dimension of Julia sets}
\author{James Waterman}
\address{School of Mathematics and Statistics, The Open University, Walton Hall, Milton Keynes MK7 6AA, UK}
\curraddr{}
\email{james.waterman@open.ac.uk}
\date{}
\begin{abstract}
We show that the Hausdorff dimension of the set of points of bounded orbit in the Julia set of a meromorphic map with a simply connected direct tract and a certain restriction on the singular values is strictly greater than one. This result is obtained by proving new results related to Wiman--Valiron theory.
\end{abstract}
\maketitle 
\section{Introduction}
Let $f: \mathbb{C} \rightarrow \mathbb{C}$ be a transcendental meromorphic function and denote the $n$th iterate of $f$ by $f^n$ for $n=1,2,3,\ldots$. The \textit{Fatou set} $F(f)$, is defined to be the subset of $\mathbb{C}$ where the iterates $f^n$ of $f$ form a normal family. The \textit{Julia set} is defined to be its complement. Finally, the \textit{escaping set} $I(f)$ is the set of $z \in \mathbb{C}$ for which $f^n(z) \rightarrow \infty$ as $n \rightarrow \infty$ and we denote by $K(f)$ those points in $\mathbb{C}$ with bounded orbit. An introduction to the Fatou set, Julia set, and escaping set and their properties can be found in the survey article of Bergweiler, see \cite{B93}. More information on $K(f)$ can be found in \cite{BergweilerBounded} and \cite{OsborneBounded}.

The Hausdorff dimension of the Julia set, $\dim J(f)$, of transcendental entire and meromorphic functions has been widely studied. Misiurewicz \cite{M81} proved that the Julia set of the exponential function is the entire plane, and hence  $\dim J(f) =2$ in this case. Major work was done by McMullen \cite{McMullen87}, who gave examples for which $\dim J(f)\cap I(f) =2$ but the Julia set is not the entire plane. It was shown by Baker \cite{Baker75} that the Julia set of every transcendental entire function contains a continuum and hence $\dim J(f) \geq 1$. It was further shown by Stallard \cite{StallardDimensions} that for each $d\in (1,2)$ there exists a transcendental entire function for which $\dim J(f) = d$.


More is known if one restricts to a specific class of entire functions.   Many results have been proven for functions in the Eremenko--Lyubich class $\mathcal{B}$. For example, if $f$ is a function of finite order in the class $\mathcal{B}$, then $I(f)$ ($\subset J(f)$) has Hausdorff dimension $2$; see \cite{FiniteOrderBaranski} and \cite{FiniteOrderSchubert}.  Stallard \cite{StallardDimensionsB} proved that for all entire functions in the Eremenko--Lyubich class $\mathcal{B}$, $\dim J(f)>1$. This was then improved by Bara\'nski, Karpi\'nska, and Zdunik \cite{BKZ09} who showed that, for a meromorphic function with a logarithmic tract, the Hausdorff dimension of the set of points  in the Julia set with bounded orbit is strictly greater than $1$. While Stallard's result made use of escaping points, Bara\'nski, Karpi\'nska, and Zdunik used completely different methods of proof based on points in $J(f) \cap K(f)$.

Until recently it was an open question as to whether there existed an entire function $f$ for which the Hausdorff dimension of the Julia set of $f$ was equal to $1$. However, Bishop \cite{BishopDim1} has constructed an entire function $f$ with $\dim J(f) = 1$. This function has multiply connected wandering domains, and so its direct tract, defined in \Cref{Wiman Valiron}, does not have an unbounded boundary component. This suggests the question of whether Bara\'nski, Karpi\'nska, and Zdunik's result can be generalized to show that, for a transcendental meromorphic function $f$ with a direct tract with an unbounded boundary component, we have that
$\dim J(f) >1$. We prove this in several  situations, and specifically in the case where the direct tract is simply connected and there is an additional restriction on the singular values associated with the direct tract.
\begin{theorem} \label{DimensionThm}
Let $f$ be a transcendental meromorphic function with a simply connected direct tract $D$.  Suppose that there exists $\lambda>1$ such that for arbitrarily large $r$ there exists an annulus $A({r}/{\lambda}, \lambda r)$ containing no singular  values of the restriction of $f$ to $D$. Then $\dim J(f)  \cap  K(f)  > 1$.
\end{theorem}
Note that there are relatively few results concerning the set $K(f)$ for transcendental entire functions. Following the proof of \cite{BKZ09}, Bergweiler showed that $\dim J(f) \cap K(f)$ is not greater than one in general. Indeed, he proved in \cite{BergweilerBounded} that  there exist transcendental entire functions for which $\dim K(f)$ can be arbitrarily small. Note that Bergweiler's examples all have direct tracts with no unbounded boundary component.

Our proof of \Cref{DimensionThm} partially follows  the approach in Bara\'nski, Karpi\'nska, and Zdunik \cite{BKZ09}, with some significant differences. The hypotheses in \cite{BKZ09} specify a logarithmic tract, for which one can make use of the logarithmic transform and the resulting expansion estimate in the tract. However, in our case, we no longer have this expansion estimate throughout the direct tract. Instead, our proof uses Wiman--Valiron theory, which gives the existence of a disc in the direct tract, around certain points where the function takes its maximum modulus, within which the function behaves like a polynomial. We use tools developed by Bergweiler, Rippon, and Stallard \cite{Tracts} for general direct tracts, and apply them to direct tracts with certain properties. These results trace back to Wiman--Valiron theory \cite{Hayman74} and Macintyre's theory of flat regions \cite{M38}. We show that in our case a larger disc than that given by  Wiman--Valiron theory lies in the direct tract, and on this disc an estimate of the function holds which is somewhat weaker than the Wiman--Valiron estimate. While this estimate on the new disc is weaker than the Wiman--Valiron estimate, it enables us to cover a much larger annulus than does the Wiman--Valiron estimate many times over. These results are of independent interest.

The paper is organized as follows. \Cref{TractsSec}  introduces the notion of a direct tract and the classification of  singularities. In \Cref{Wiman Valiron}, we give our results on the size of enlarged Wiman--Valiron discs. We then use these results in \Cref{Hausdorff} in order to prove a theorem on the Hausdorff dimension of certain Julia sets, of which \Cref{DimensionThm} is a special case. Finally, in \Cref{Examples} we give two nontrivial examples of transcendental entire functions to which we can apply our results.
\section{Direct tracts and the logarithmic transform}\label{TractsSec}
First, we recall the classification of the singularities of the inverse function due to Iversen \cite{Iversen}, as well as the definition of a tract, following terminology found in \cite{Tracts}. Let $f$ be a meromorphic function and consider $a \in \widehat{\mathbb{C}}$. For $R>0$, let $U_R$ be a component of $f^{-1}(D(a,R))$ (where $D(a,R)$ is the
open disc centered at $a$ with radius $R$ with respect to the spherical metric) chosen so that $R_1< R_2$ implies that $U_{R_1} \subset U_{R_2}$. Then either $\bigcap_{R} U_R = \{z\}$ for some unique $z \in \mathbb{C}$ or $\bigcap_{R} U_R = \emptyset$. 

In the first case, $a=f(z)$ and $a$ is an \textit{ordinary point} if $f'(z) \neq 0$, or $a$ is a \textit{critical value} if $f'(z)=0$. If $f'(z)=0$, we call $z$ a \textit{critical point}. In the second case,  $f$ has a \textit{transcendental singularity} over $a$. The transcendental singularity is called \textit{direct} if $f(z) \neq a$ for all $z \in U_R$, for some $R>0$. Otherwise it is \textit{indirect}. Further, a direct singularity is called \textit{logarithmic} if $f: U_R \rightarrow D(a,R) \setminus \{a\}$ is a universal covering. These components $U_R$ are called \textit{tracts} for $f$. Note that the structure of $U_R$ depends on $R$. We will mainly restrict $f$ to a specific tract, which motivates the following definition; see for example \cite{Tracts}.

\begin{definition}
Let $D$ be an unbounded domain in $\mathbb{C}$ whose boundary consists of piecewise smooth curves and suppose that the complement of $D$ is unbounded. Further, let $f$ be a complex valued function whose domain of definition contains the closure $\overline{D}$ of $D$. Then $D$ is called a \textit{direct tract} of $f$ if $f$ is holomorphic in $D$, continuous in $\overline{D}$, and if there exists $R>0$ such that $|f(z)|=R$ for $z \in \partial D$ while $|f(z)| > R$ for $ z \in D$. If, in addition, the restriction $f:D \rightarrow \{z \in \mathbb{C}: |z| > R\}$ is a universal covering, then $D$ is called a \textit{logarithmic tract}. Further, we call $R$ the \textit{boundary value} of the direct tract.
\end{definition}
Of note is that meromorphic functions with a direct singularity have a direct tract. Further, every transcendental entire function has a direct tract. Although, for a meromorphic function, the boundaries of these tracts are piecewise analytic, the geometry of these tracts can be incredibly varied and sometimes quite wild. Direct tracts need not be simply connected or even have an unbounded boundary component. However, logarithmic tracts are always simply connected and direct tracts in the class $\mathcal{B}$ are logarithmic for a sufficiently large $R$ in the above. 

Logarithmic tracts have been of great use in the iteration of entire and meromorphic functions with many major results and constructions such as \cite{RRRS} and \cite{Tracts} making use of their properties and the additional tools that they give. A technique to identify a logarithmic tract was given in \cite{IDingtracts}.  The major tool used is the logarithmic transform which was first studied in the context of entire functions by Eremenko and Lyubich \cite{EL92}. We now introduce the logarithmic transform, $F$, in the setting of a direct tract with an unbounded boundary component. While the logarithmic transform has been studied in many papers on transcendental entire dynamics, it has  almost always been applied to a logarithmic tract.

Let $f$ be a meromorphic function with a direct tract $D$ with an unbounded complementary component; any simply connected direct tract must have such a complementary component. By  performing a translation if necessary, assume further that $0$ is contained in an unbounded complementary component of $D$. Then, denote by $F$ the logarithmic transform of $f$; that is, ${\exp \circ F = f \circ \exp}$. Note that $F$ is periodic with period $2 \pi i$ and maps each component of $\log D$ into a right half-plane. Unlike for a logarithmic tract, $F$ is not necessarily univalent on each component of  $\log D$. However, as there exists an unbounded complementary component, one may still lift $f$ by the branches of the logarithm.

By applying the Koebe $1/4$-theorem, we are able to obtain an estimate on the entire function  in the direct tract on a domain on which the function is univalent. This result is a modification of an expansion estimate due to Eremenko and Lyubich \cite{EL92} and is similar to the approach of Rippon and Stallard in \cite[Lemma 2.5]{Bakerdomains}.

\begin{lemma}\label{expansion}
Let $f$ be a transcendental meromorphic function with a direct tract $D$, and let $\lambda>1$. Suppose, for some $r>0$, that the restriction of $f$ to $D$ has no singular values in $A(r/\lambda,\lambda r)$ and $ \gamma \subset D$ is a simple unbounded curve on which $|f(z)|=r$. Then, 
\begin{enumerate}[(a)]
\item \label{expansiona}the function $\Phi(z)=\log f(z)$, for $z \in D$, is univalent on a simply connected domain $\Omega$ such that $\gamma \subset \Omega \subset D$, and 
\[\Phi(\Omega)=S=\{w:\log (r/\lambda)<\operatorname{Re} w < \log \lambda r\};\]
\item \label{expansionb}the function $F(w)=\Phi(e^w)=\log f(e^w)$ maps each component of $\log \Omega$ univalently onto $S$;
\item \label{expansionc}any analytic branch $H$ of $F^{-1}:S \rightarrow \log D$ satisfies 
\begin{equation}\label{HDerbound}
|H'(w)|\leq \frac{4 \pi}{\operatorname{dist}(w, \partial S)}, \quad \text{for}~ w \in S,
\end{equation} 
or equivalently
\begin{equation}\label{fDerbound}
\left|\frac{zf'(z)}{f(z)}\right| \geq 
\begin{cases}
\frac{1}{4 \pi}  \log |\lambda r/f(z)|, \quad \text{for}~ |f(z)|\geq r\\
\frac{1}{4 \pi}  \log |\lambda f(z)/r|, \quad \text{for}~ |f(z)|\leq r,
\end{cases}
\end{equation}
for $z\in \Omega$.
\end{enumerate}

\end{lemma}
\begin{proof}
Take a base point $z_0\in \gamma$ and let $w_0=\log f(z_0)$. Then, let $\Psi$ denote the inverse branch of $ z\mapsto \Phi(z)=\log f(z)$ such that $\Psi(w_0)=z_0$. By our hypothesis on the singular values of $f$, the branch $\Psi$ can be analytically continued along every curve in $S$ starting from $w_0$, with image values in $D$. Since $S$ is simply connected, we deduce from the monodromy theorem that $\Psi$ extends to be analytic throughout $S$, and $\Omega := \Psi(S) \subset D$.

Two cases can then arise \cite[p. 283]{Nevanlinna}: either $\Psi$ is univalent in $S$ or $\Psi$ is periodic in $S$, with period $2 \pi i m$, where $m \in \mathbb{N}$. In the latter case, however, $\Psi(S)$ must be bounded, which is impossible, since $\gamma \subset \Psi(S)=\Omega$.

Now, we use the fact that $H=\log \Psi$, for some branch of the logarithm, and hence $H$ maps $S$ univalently onto a domain that contains no disc of radius greater than $\pi$. Therefore, \eqref{HDerbound} holds by the Koebe $1/4$-theorem, and \eqref{fDerbound} follows by the change of variables $z=e^w$.
\end{proof}

\section{Wiman--Valiron theory in tracts}\label{Wiman Valiron}

First, we introduce some preliminary concepts, discuss the original Wiman--Valiron result of Bergweiler, Rippon, and Stallard, and give some basic lemmas; see \cite{Tracts}. 

 Let $v:\mathbb{C} \rightarrow [0,\infty)$ be a non-constant subharmonic function. Then the function
\[B(r,v) = \max_{|z|=r} v(z)\] 
is increasing, convex with respect to $\log r$, and $B(r,v)/\log r\rightarrow \infty$ as $r\rightarrow \infty$. So,
\[a(r,v)= \frac{dB(r,v)}{d\log r}=rB'(r,v)\]
exists except for possibly a countable set of $r$ values and is non-decreasing.

Let $f$ be a meromorphic function with a direct tract $D$ with an unbounded boundary component and boundary value $R$. Define 
\begin{equation}\label{v}
v(z)=\log \frac{|f(z)|}{R},
\end{equation}
 for $z \in D$ and $v(z)=0$ elsewhere; 
then $B(r,v)= \max_{|z|= r} \log |f(z)|/R$.

It was proved in \cite{Tracts} that any direct tract of $f$ contains a disc about points where the maximum modulus is attained, on which $f$ behaves asymptotically like a polynomial. As we will refer back to this result several times, we state it here.
\begin{theorem}[\cite{Tracts}, Theorem 2.2] \label{BRS}
Let $D$ be a direct tract of $f$ and let $\tau > \frac{1}{2}$. Let $v$ be defined as in \eqref{v} and let $z_r \in D$ be a point satisfying $|z_r|=r$ and $v(z_r)=B(r,v)$. Then there exists a set $E \subset [1, \infty)$ of finite logarithmic measure such that if $r \in  [1, \infty) \setminus E$, then $D(z_r, r/a(r,v)^\tau) \subset D$. Moreover,
\[f(z) \sim \left(\frac{z}{z_r}\right)^{a(r,v)} f(z_r), \quad \text{for} ~ z \in D(z_r, r/ a(r,v)^\tau),\]
as $r \rightarrow \infty$, $r\notin E$.
\end{theorem}

We now show that a larger disc than given by \Cref{BRS} is possible inside a simply connected direct tract. In a general direct tract this is not possible, as shown by an example due to  Bergweiler \cite{SizeOfTracts}. Note that, Bergweiler's example has a direct tract with no unbounded boundary component.

First, we recall a standard estimate on $a(r,v)$; see \cite[Lemma 6.10]{Tracts}.
\begin{lemma}\label{Epsilon}
Let $v: \mathbb{C}\rightarrow [-\infty, \infty)$ be subharmonic and let $\epsilon > 0$. Then there exists a set $E\subset [1, \infty)$ of finite logarithmic measure such that 
\[a(r,v) \leq B(r,v)^{1+\epsilon}\]
for $r\geq 1$, $r \notin E$.
\end{lemma}
Next, we state a more general version of \cite[Lemma 11.2]{Tracts}, giving a slightly different estimate of $B(r,v)$ for a general subharmonic function which is  suited for our purpose. This statement is extracted from the inequalities given in the proof of \cite[Lemma 11.2]{Tracts} (in which the function $\Phi(x)= B(r,v)$ and $\Phi'(x)= a(r,v)$, with $r=e^x$).
\begin{lemma}\label{Growth}
Let $v\rightarrow [-\infty, \infty)$ be subharmonic and let $\beta >\alpha> 0$. Then there exists a set $E\subset [1, \infty)$ of finite logarithmic measure such that 
\[B(s,v) \leq B(r,v)+a(r,v) \log \frac{s}{r} + a(r,v)^{1-\alpha-\beta}, \quad \text{for} ~ \left|\log \frac{s}{r} \right| \leq \frac{1}{a(r,v)^\beta}, r \notin E,\]
uniformly as $r \rightarrow \infty$.

\end{lemma}

We shall use \Cref{Growth} to prove a generalization of \Cref{BRS}, which provides a much larger disc in a direct tract that satisfies certain extra conditions, at the expense of less control on the function as we move further out from the original Wiman--Valiron disc. First, if there exists a direct tract surrounding a sparse collection of zeros, we show that a larger disc exists inside the direct tract than given by \Cref{BRS}. Note that this hypothesis is satisfied by a simply connected direct tract.
\begin{theorem}\label{WV}
Let $f$ be a meromorphic function with a direct tract $D$ with an unbounded boundary component. Let $v$ be defined as in \eqref{v}, and let $z_r \in D$ be a point for which $|z_r|=r$ and $v(z_r)= B(r,v)$. Suppose that there exists $\lambda>1$ and a set $E \subset[1, \infty)$ of finite logarithmic measure such that, for $r \in [1,\infty)\setminus E$, we have that $z_r \in L_r$, a simply connected component of $A(r/\lambda,\lambda r)\cap D$. Also, let $\frac{1}{2}>\tau>0$. Then, for $r \in [1,\infty)\setminus E'$, where $E'$ has finite logarithmic measure, there exists $D(z_r, r/ a(r,v)^\tau) \subset D$.
\end{theorem}
\begin{proof}
Choose $\alpha$ and $\beta$ such that $0<\alpha<\beta<\tau<1/2$ and $1-\alpha-\beta=\sqrt{1-2\tau}=\xi(\tau)$, say. Let $E'$ be the union of $E$ and the exceptional sets in \Cref{Epsilon} and \Cref{Growth} for these values of $\alpha$ and $\beta$, where $\epsilon = \beta$. Further, set $\rho = 2 r a(r,v)^{-\tau}$, where $r$ is large enough that $a(r,v) \neq 0$. Consider the subharmonic function
\begin{equation}\label{Ubound}
u(z)= v(z) - B(r,v)-a(r,v) \log \frac{|z|}{r} \leq a(r,v)^{\xi(\tau)},
\end{equation}
for $r \notin E'$ and $z\in \overline{D}(z_r,512 \rho)$ by \Cref{Growth}, and the fact that for $z \in \overline{D}(z_r,512 \rho)$,
\[\left |\frac{z-z_r}{z_r} \right|\leq \frac{512\rho}{r}= \frac{1024}{a(r,v)^\tau}=o(1) ~\text{as}~ r \rightarrow \infty,\]
since $\lim_{r\rightarrow \infty} a(r,v)=\infty$, and so, for large $r$,
\begin{equation}\label{SizeBound}
\left|\log \frac{|z|}{r} \right| = \left| \log \left| 1 + \frac {z-z_r}{z_r}\right|\right| \leq 2 \left |\frac{z-z_r}{z_r} \right| \leq \frac{2048}{a(r,v)^\tau}\leq \frac{1}{a(r,v)^\beta}.
\end{equation}

Now, we show that $D(z_r, r/ a(r,v)^\tau) \subset D$ for $r \notin E'$ sufficiently large. Following \cite[p. 395--396]{Tracts}, suppose not. First, by assumption, $ D(z_r,256 \rho)$  does not contain any bounded complementary components, since the exceptional set of $r$ has finite logarithmic length and $L_r$ is assumed simply connected.  Then, since $D$ has an unbounded boundary component, there exists an unbounded component of the complement of $D$ that intersects $\partial D(z_r,t)$ for $\rho \leq t \leq 256 \rho$. Let $V$ be the component of $D \cap D(z_r, 512 \rho)$ that contains $z_r$ and let $\Gamma= \partial V \cap \partial D(z_r,512 \rho)$. Note that $V$ is simply connected, as $L_r$ is simply connected. Further, in this case, let $t \theta^*(z_r,t)$ be the linear measure of the intersection of $\partial D(z_r,t)$ with $D$. Then, $\theta^*(z_r,t)\leq 2 \pi$ for $\rho \leq t \leq 256 \rho$, since $\partial D(z_r,t)$ is not wholly contained in $D$. Hence, by a result of Tsuji; see \cite[p. 112]{Tsuji},
\[\omega(z_r,\Gamma,V)\leq 3 \sqrt{2} \exp \left(- \pi \int_\rho^{256\rho} \frac{dt}{t \theta^*(z_r,t)}\right) \leq 3 \sqrt{2}\exp \left(- \frac{1}{2} \int_\rho^{256\rho} \frac{dt}{t}\right)=\frac{3\sqrt{2}}{2^4} < \frac{1}{2},\]
where $\omega(z_r,\Gamma,V)$ denotes the harmonic measure of $\Gamma$ at $z_r$ in $V$.
Thus, for $\Sigma=\partial V \setminus \Gamma$ we have
\[\omega(z_r,\Sigma,V)= 1- \omega(z_r,\Gamma,V) > \frac{1}{2}.\]

Next, for $z \in \Gamma$, we have $v(z)=0$, since $z$ is on the boundary of $D$, so from \eqref{Ubound}, \eqref{SizeBound}, and \Cref{Epsilon} with $\epsilon = \beta$, if $r \notin E'$ is sufficiently large, then 
\begin{align*}
u(z)&= -B(r,v) -a(r,v) \log \frac{|z|}{r}\\
& \leq - B(r,v) + a(r,v)^{1-\beta}\\
&\leq  - B(r,v) + B(r,v)^{(1-\beta)(1+\epsilon)}\\
&\leq - \frac{1}{2} B(r,v).
\end{align*}
Hence, by \eqref{Ubound} and the above, we can apply the two constant theorem (see for example \cite[Theorem 4.3.7]{Ransford}) for large $r \notin E'$ and obtain 
\begin{align*}
u(z_r) &\leq -\frac{1}{2} \omega(z_r,\Sigma,V) B(r,v) + a(r,v)^{\xi(\tau)}(1-\omega(z_r,\Sigma,V))\\
&\leq -\frac{1}{4} B(r,v)+ a(r,v)^{\xi(\tau)} \\
&\leq  -\frac{1}{4} B(r,v)+ B(r,v)^{\xi(\tau)(1+\epsilon)} \\
&\leq -\frac{1}{8} B(r,v),
\end{align*}
which gives a contradiction, since $u(z_r)=0$, while $B(r,v) \rightarrow \infty$ as $r \rightarrow \infty$.
Hence,  $D(z_r, r/ a(r,v)^\tau) \subset D$ for $r \notin E'$ sufficiently large.
\end{proof}

We now show that, given a disc of the form $D(z_r, r/ a(r,v)^\tau)$, for $0<\tau<1/2$, inside our tract, we can obtain an estimate for the function $f$ inside that disc. This estimate  includes the case of the original Wiman--Valiron estimate when $\tau>1/2$. A larger disc than that given by Wiman--Valiron theory introduces a worse error to the estimate the larger the new disc is. Note that for the following theorem, we do not insist on any additional hypotheses on the tract, such as being simply connected; we just require the existence of a suitably large disc inside the direct tract. 
Further, note that the function $\xi(\tau) = \sqrt{1-2\tau}$ is chosen so that $\xi(\tau) \rightarrow 1$ as $\tau \rightarrow 0$ and $\xi(\tau) \rightarrow 0$ as $\tau \rightarrow 1/2$, while also $\xi(\tau) < 1-\tau$. These properties will be important in this and subsequent proofs.

\begin{theorem}\label{WVEstimate}
Let $f$ be a meromorphic function with a direct tract $D$ and $\tau>0$. Further, let $v$ be defined as in \eqref{v}, and let $z_r \in D$ be a point for which $|z_r|=r$ and $v(z_r)= B(r,v)$. Then there exists a set $E\subset [1, \infty)$ of finite logarithmic measure, such that, if there exists a disc $ D(z_r, r/ a(r,v)^\tau) \subset D$ for $r \notin E$ sufficiently large, then there exists an analytic function $g$ in $D(z_r, r/ a(r,v)^\tau)$ such that 
\[\log f(z)=\log f(z_r) + a(r,v) \log \frac{z}{z_r} + g(z), \quad \text{for}~ z \in D(z_r, r/ a(r,v)^\tau),\]
where
 \[g(z)=
\begin{cases}
O(a(r,v)^{\xi(\tau)}) &\text{for}~ z\in D(z_r, r/ a(r,v)^\tau) ~\text{and} ~ \tau <1/2, \\
o(1) &\text{for}~ z\in D(z_r, r/ a(r,v)^{\tau})~\text{and} ~ \tau>1/2,
\end{cases}
\] and $\xi(\tau)=\sqrt{1-2\tau}$ as $r\rightarrow \infty$, $r \notin E$.

\end{theorem}
\begin{proof}
Set 
\begin{equation}\label{geq}
g(z) = \log \frac{f(z)}{f(z_r)} - a(r,v) \log \frac{z}{z_r},
\end{equation}
where the branches of the logarithms are chosen so that $g(z_r)=0$. By the Borel--Carath\'eodory inequality \cite[p. 20]{valiron49},
\begin{equation}\label{g}
\max_{|z-z_r| \leq t} |g(z)| \leq 4 \max_{|z-z_r| \leq 2t}\operatorname{Re} g(z) \leq 4 a(r,v)^{\xi(\tau)},
\end{equation}
for $0 < t <\rho/2$ and $r \notin E$, by \Cref{Growth} with $\alpha$ and $\beta$ chosen so that $\beta> \alpha>0$ and $1-\alpha-\beta=\xi(\tau)$.

Thus, by $\eqref{geq}$, for $z\in D(z_r, ra(r,v)^{-\tau})$ and $r\notin E$,
\[\log f(z)=\log f(z_r) + a(r,v) \log \frac{z}{z_r} + g(z), \quad\text{for}~ z \in D(z_r, r/ a(r,v)^\tau),\]
where  $g(z) = O(a(r,v)^{\xi(\tau)})$ for $\tau <1/2$ by \eqref{g} and $g(z) = o(1)$ inside $D(z_r, r/ a(r,v)^{\tau})$ for $\tau>1/2$ by \Cref{BRS}. 
\end{proof}
Note that in the above result, the case $0<\tau<1/2$ is new. Further, for $\tau=1/2$ it is natural to ask what the error, $g$, in the Wiman--Valiron estimate is and whether it is possibly $O(1)$. However, this proof does not show it.

We now use \Cref{WVEstimate} in order to obtain an estimate on the size of the image of this new disc. First, we label three sets where we choose the principal branch of the logarithm; that is the branch corresponding to the principal value of the argument. Let
\begin{equation}\label{SBox}
S_r= \left\{w: |\operatorname{Re} w - \log r | \leq \frac{1/2}{a(r,v)^\tau}, |\operatorname{Im}w-\arg z_r| \leq \frac{1/2}{a(r,v)^\tau}\right\}.
\end{equation}
Further, let
\begin{equation}\label{QBox}
Q=\left\{w:\left|\operatorname{Re} w - \log |f(z_r)| \right|< \frac{1}{8} a(r,v)^{1-\tau}, \left|\operatorname{Im} w - \arg f(z_r) \right|< \frac{1}{8} a(r,v)^{1-\tau}\right\},
\end{equation}
and 
\begin{equation}\label{QHatBox}
\hat{Q}=\left\{w:\left|\operatorname{Re} w - \log |f(z_r)| \right|< \frac{1}{4} a(r,v)^{1-\tau}, \left|\operatorname{Im} w - \arg f(z_r) \right|< \frac{1}{4} a(r,v)^{1-\tau}\right\}.
\end{equation}
\begin{theorem} \label{Covering}
 Let $f$ be a meromorphic function with a direct tract $D$ with an unbounded boundary component and $\tau$, $v$, $z_r$, and $E$  be as in \Cref{WVEstimate}. Consider the logarithmic transform $F$ of $f$. If, for $ r \in [1, \infty) \setminus E$, there exists $D(z_r, r/ a(r,v)^\tau) \subset D$, then ${S_r \subset \log D(z_r, r/ a(r,v)^\tau)}$ is mapped univalently by $F$ and $F(S_r) \supset \hat{Q} \supset Q$. 
\end{theorem}
\begin{proof}
Let $w_r= \log z_r + i \arg z_r$, where we choose the principal branch of the logarithm and recall that $F$ is the logarithmic transform of $f$.
From \Cref{WVEstimate},
\begin{equation}\label{FWVEstimate}
F(w) = F(w_r) + a(r,v) (w-w_r) + O(a(r,v)^{\xi(\tau)}),
\end{equation}
for $w \in \log D(z_r, r/ a(r,v)^\tau)$ and as $r \rightarrow \infty$.

\begin{figure}
\begin{tikzpicture}[scale=1]
\filldraw [ fill=gray,opacity=0.2](-3,1) -- (-4,1) -- (-4,2) -- (-3,2) -- cycle;
\draw (-3,1) -- (-4,1) -- (-4,2) -- (-3,2) -- cycle;
\draw (-3,1) -- (-4,1) -- (-4,2) -- (-3,2) -- cycle;
\draw (0,0) -- (4,0) -- (4,4) -- (0,4) -- cycle;
\draw (0.9,0.9) -- (3.1,0.9) -- (3.1,3.1) -- (0.9,3.1) -- cycle;
\draw (1.4,1.4) -- (2.6,1.4) -- (2.6,2.6) -- (1.4,2.6) -- cycle;
\filldraw [ fill=gray,opacity=0.2](0,-0.5) to  [out=180,in=-90] (-0.5,0) to  [out=90,in=-90] (0,1) to  [out=90,in=-90] (-0.3,2) to  [out=90,in=-135] (0.32,3.5) to  [out=45,in=180] (1.5,4.2) to  [out=0,in=180] (2.5,3.75) to  [out=0,in=180] (3.7,4.25) to  [out=0,in=45] (4.1,3.6) to  [out=-135,in=90] (4.3,1.5) to  [out=-90,in=90] (3.9,1.0) to  [out=-90,in=60] (4.3,-0.4) to  [out=240,in=0] (3,0.25) to  [out=180,in=0] (2.2,-0.6) to  [out=160,in=0] (0,-0.5);
\draw (0,-0.5) to  [out=180,in=-90] (-0.5,0) to  [out=90,in=-90] (0,1) to  [out=90,in=-90] (-0.3,2) to  [out=90,in=-135] (0.32,3.5) to  [out=45,in=180] (1.5,4.2) to  [out=0,in=180] (2.5,3.75) to  [out=0,in=180] (3.7,4.25) to  [out=0,in=45] (4.1,3.6) to  [out=-135,in=90] (4.3,1.5) to  [out=-90,in=90] (3.9,1.0) to  [out=-90,in=60] (4.3,-0.4) to  [out=240,in=0] (3,0.25) to  [out=180,in=0] (2.2,-0.6) to  [out=160,in=0] (0,-0.5);
\draw[->] (-2.5,1.5) -- (-0.5,1.5);
\node at  (-1.5,1.5) [above]{$F$} ;
\draw[<->] (-3,0.75) -- (-4,0.75);
\node at (-3.5,0.75) [below]{$\frac{1}{a(r,v)^\tau}$};
\node at (-3.5,2) [above]{$S_r$};

\node at (0.2,3.7) {$\tilde{Q}$};
\node at (1.1,2.8) {$\hat{Q}$};
\node at (1.59,2.34) {$Q$};
\filldraw 
(2,2) circle (1pt) node[below]{$F(w_r)$};

\draw[<-] (5,0) -- (5,1.75);
\draw[->] (5,2.25) -- (5,4); 
\node at (5.2,2) {$a(r,v)^{1-\tau}$};
\node at (-1,-0.75) [below]{$\partial F(S_r)$};
\draw[->] (-1,-0.75) -- (-0.4,-0.4);
\end{tikzpicture}
\caption{A sketch of the argument in the proof of \Cref{Covering}.}
\label{FigureCovering}
\end{figure}
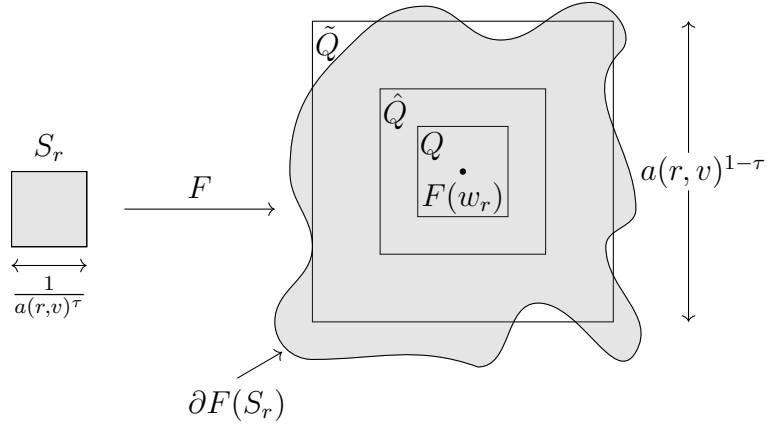

The domain 
\begin{align*}
S_r=& \left\{w: |\operatorname{Re} w - \log r | \leq \frac{1/2}{a(r,v)^\tau}, |\operatorname{Im}w-\arg z_r| \leq \frac{1/2}{a(r,v)^\tau}\right\}\\
\subset&  \log D(z_r,  r/ a(r,v)^\tau)
\end{align*}
is a square of side length $1/a(r,v)^\tau$ centered at $w_r$ and mapped by \[w \mapsto F(w_r)+a(r,v)(w-w_r)\] to a square, $\tilde{Q}$, of side $a(r,v)^{1-\tau}$ centered at $F(w_r)$. Since \[F(w) = F(w_r) + a(r,v) (w-w_r) + O(a(r,v)^{\xi(\tau)}),\] the image of $\partial S_r$ under $F$ is a closed curve lying entirely outside of a square of side \[a(r,v)^{1-\tau}-O(a(r,v)^{\xi(\tau)})\] centered at $F(w_r)$  for large enough $r$. This closed curve winds exactly once around $F(w_r)$, since ${{1-\tau}>\sqrt{1-2\tau}={\xi(\tau)}}$. Hence, by the argument principle, $F$ is univalent on $S_r$ and $F(S_r)$ contains a square of side length  greater than $\frac{1}{2}a(r,v)^{1-\tau}$. \end{proof}

\section{Hausdorff dimension} \label{Hausdorff}
In this section, we show that the set of points of bounded orbit in the Julia set of a meromorphic map with a direct tract with conditions on the singular values, zeros, and the boundary of the tract has Hausdorff dimension strictly greater than one. \Cref{DimensionThm} follows from \Cref{Dim}, below, together with \Cref{WV}, replacing $\lambda$ by $\lambda/2$ in \Cref{Dim} and using the fact that the exceptional set has finite logarithmic length.
 Recall that $v$ is defined as in \eqref{v}, and $z_r$ is a point for which $|z_r|=r$ and $v(z_r)= B(r,v)$.

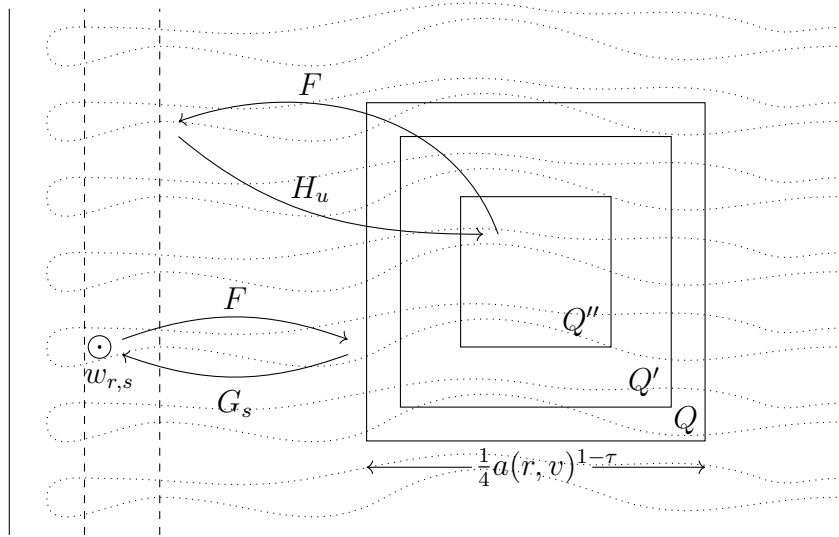
\begin{figure}
\begin{tikzpicture}[scale=1]
\draw (-0.25,-0.25) -- (4.25,-0.25) -- (4.25,4.25) -- (-0.25,4.25) -- cycle;
\draw (1,1) -- (3,1) -- (3,3) -- (1,3) -- cycle;
\draw (0.2,0.2) -- (3.8,0.2) -- (3.8,3.8) -- (0.2,3.8) -- cycle;
\draw (-5,-1.5)--(-5,5.5);
\draw[dashed] (-4,-1.5)--(-4,5.5);
\draw[dashed] (-3,-1.5)--(-3,5.5);
\draw (-3.8,1) circle [radius=0.15];
\node at (-3.67,0.85) [below]{$w_{r,s}$};
\filldraw  (-3.8,1) circle (0.5pt);

\node at (4.3,-0.3)[above left]{$Q$};
\node at (3,1)[above left]{$Q''$};
\node at (3.8,0.2)[above left]{$Q'$};

\draw [->] (-3.5,1.1) to  [out=20,in=160] (-0.5,1.1);
\node at (-2,1.65){$F$};

\draw [->] (-0.5,0.9) to  [out=200,in=-20] (-3.5,0.9);
\node at (-2,0.25){$G_s$};

\draw [->] (1.5,2.5) to  [out=110,in=20] (-2.75,4);
\node at (-1,4.5){$F$};

\draw [->] (-2.75,3.8) to  [out=-40,in=180] (1.3,2.5);
\node at (-1,3.1){$H_u$};

\draw[dotted] (6,0) to  [out=180,in=-20] (3,0) to  [out=160,in=30] (0,0) to  [out=210,in=0] (-3,0) to  [out=180,in=-90] (-4.5,0) to  [out=90,in=180] (-3.7,0.25)to  [out=0,in=190] (-1,0.25) to  [out=10,in=170] (2,0.5) to  [out=-10,in=160] (5,0.25) to  [out=-20,in=180] (6,0.25);

\begin{scope}[shift={(-0,1)}]
   \draw[dotted] (6,0) to  [out=180,in=-20] (3,0) to  [out=160,in=30] (0,0) to  [out=210,in=0] (-3,0) to  [out=180,in=-90] (-4.5,0) to  [out=90,in=180] (-3.7,0.25)to  [out=0,in=190] (-1,0.25) to  [out=10,in=170] (2,0.5) to  [out=-10,in=160] (5,0.25) to  [out=-20,in=180] (6,0.25);
  \end{scope}
  \begin{scope}[shift={(-0,2)}]
   \draw[dotted] (6,0) to  [out=180,in=-20] (3,0) to  [out=160,in=30] (0,0) to  [out=210,in=0] (-3,0) to  [out=180,in=-90] (-4.5,0) to  [out=90,in=180] (-3.7,0.25)to  [out=0,in=190] (-1,0.25) to  [out=10,in=170] (2,0.5) to  [out=-10,in=160] (5,0.25) to  [out=-20,in=180] (6,0.25);
  \end{scope}
  \begin{scope}[shift={(-0,3)}]
   \draw[dotted] (6,0) to  [out=180,in=-20] (3,0) to  [out=160,in=30] (0,0) to  [out=210,in=0] (-3,0) to  [out=180,in=-90] (-4.5,0) to  [out=90,in=180] (-3.7,0.25)to  [out=0,in=190] (-1,0.25) to  [out=10,in=170] (2,0.5) to  [out=-10,in=160] (5,0.25) to  [out=-20,in=180] (6,0.25);
  \end{scope}
  \begin{scope}[shift={(-0,4)}]
   \draw[dotted] (6,0) to  [out=180,in=-20] (3,0) to  [out=160,in=30] (0,0) to  [out=210,in=0] (-3,0) to  [out=180,in=-90] (-4.5,0) to  [out=90,in=180] (-3.7,0.25)to  [out=0,in=190] (-1,0.25) to  [out=10,in=170] (2,0.5) to  [out=-10,in=160] (5,0.25) to  [out=-20,in=180] (6,0.25);
  \end{scope}
  \begin{scope}[shift={(-0,5)}]
   \draw[dotted] (6,0) to  [out=180,in=-20] (3,0) to  [out=160,in=30] (0,0) to  [out=210,in=0] (-3,0) to  [out=180,in=-90] (-4.5,0) to  [out=90,in=180] (-3.7,0.25)to  [out=0,in=190] (-1,0.25) to  [out=10,in=170] (2,0.5) to  [out=-10,in=160] (5,0.25) to  [out=-20,in=180] (6,0.25);
  \end{scope}
    \begin{scope}[shift={(-0,-1)}]
   \draw[dotted] (6,0) to  [out=180,in=-20] (3,0) to  [out=160,in=30] (0,0) to  [out=210,in=0] (-3,0) to  [out=180,in=-90] (-4.5,0) to  [out=90,in=180] (-3.7,0.25)to  [out=0,in=190] (-1,0.25) to  [out=10,in=170] (2,0.5) to  [out=-10,in=160] (5,0.25) to  [out=-20,in=180] (6,0.25);
  \end{scope}

\draw[<-] (-0.25,-0.6 ) -- (1.15,-0.6);
\draw[->] (2.75,-0.6 ) -- (4.25,-0.6); 
\node at (2.15,-0.6) {$ \frac{1}{4} a(r,v)^{1-\tau}$};

\end{tikzpicture}
\caption{Construction of $J_B$ in \Cref{Dim}.}
\label{Figure}
\end{figure}

\begin{theorem} \label{Dim}
Let $f$ be a meromorphic map with a direct tract $D$ with an unbounded boundary component.  Fix $\lambda > 1$ and $\frac{1}{2} > \tau>0$, and let $E$ be the set of finite logarithmic measure  in \Cref{{WVEstimate}}. Suppose that for arbitrarily large $r\in [1,\infty)\setminus E$ there exists an annulus $A({r}/{\lambda}, \lambda r)$ such that both
\begin{enumerate}[(i)]
\item $D(z_r, r/a(r,v)^\tau) \subset A({r}/{\lambda}, \lambda r) \cap D$, and \label{Dim1}
\item $A({r}/{\lambda}, \lambda r)$ contains no singular values of the restriction of $f$ to $D$. \label{Dim2}
\end{enumerate}
Then $\dim   J(f) \cap K(f)> 1$.
\end{theorem}
\begin{proof}

First, let $E$ be the set of finite logarithmic measure in \Cref{{WVEstimate}}. Fix $\frac{1}{2}>\tau>0$ and $\lambda>1$. Then, take $r\in[1,\infty)\setminus E$ so large that \ref{Dim1} and \ref{Dim2} are satisfied.  We use the sets $S_r$, $Q$, and $\hat{Q}$  introduced in \eqref{SBox}, \eqref{QBox}, and \eqref{QHatBox} before the statement of \Cref{Covering}. By \Cref{Covering}, the image of  $S_r \subset D(z_r, r/a(r,v)^\tau)$ under $\log f$ covers the squares $Q$ and $\hat{Q}$.

Following \cite{BKZ09} we will construct an iterated function system based on taking inverse branches of $F^2$ that map $Q$ to $Q$, where $F$ is the $2\pi i$ periodic logarithmic transform of $f$ which satisfies $\exp \circ F = f \circ \exp$; see \Cref{Figure}. We show that the dimension of the attractor of this system is strictly greater than $1$. The method for constructing this iterated function system is, however, somewhat different from that of  Bara\'nski, Karpi\'nska, and Zdunik, and we make use of the tools from Wiman--Valiron theory developed in \Cref{Wiman Valiron}.

\textbf{Step 1.} We  use our results from \Cref{{Wiman Valiron}} to construct branches of $F^{-1}$ from $Q$ to $\log D(z_r,r/a(r,v)^\tau)$. We let $w_r=\log z_r + i \arg z_r$, where $\arg$ is the principal argument. We now let $w_{r,s}= w_{r,0} + 2 \pi i s$, for $s \in \mathbb{Z}$, where $w_{r,0}=w_r$, and let $S_{r,s}=S_r+ 2 \pi i s$, for $s \in \mathbb{Z}$. 
Further, let $G_s$ be the branch of $F^{-1}$ satisfying $G_s(F(w_{r,0}))=w_{r,s}$. Then, $G_s: \hat{Q} \rightarrow S_{r,s}$ is univalent, since $F$ is univalent on $S_{r,s}$ and $F(S_{r,s}) \supset \hat{Q}\supset Q$, by \Cref{Covering}. By \eqref{FWVEstimate} and by Cauchy's estimate
\[F'(w)=a(r,v)+O(a(r,v)^{\xi(\tau)+\tau})=a(r,v)(1+o(1)), ~\text{for}~ w\in S_{r,s}, F(w) \in Q, ~\text{and}~ s \in \mathbb{Z},\]
as $r \rightarrow \infty$. 
Hence there exists an absolute constant $K>0$ such that
\begin{equation}\label{Gs estimate}
 \max_{w \in Q} |(G_s)'(w)| < \frac{K}{a(r,v)}, ~\text{for}~ s \in \mathbb{Z}.
 \end{equation}
Note that $F'(w_{r,s})=a(r,v)$, and so
\begin{equation}\label{Gwrs estimate}
 |(G_s)'(F(w_{r,s}))| = \frac{1}{a(r,v)}, ~\text{for}~ s \in \mathbb{Z}.
 \end{equation}

\textbf{Step 2.} Next, using similar methods to \cite{BKZ09}, we construct another family of branches of $F^{-1}$. First, choose $r$ so large that $\max_{w\in l} \operatorname{Re} F(w) > \log r$, where $l={\{w:\operatorname{Re}w=\log r\}}$. This is possible since 
\[\frac{\max_{w\in l} \operatorname{Re} F(w)}{\log r} = \frac{B(r,v)}{\log r} \rightarrow \infty\]
as $r\rightarrow \infty$. Since $D$ has an unbounded boundary component with boundary value $R$, any domain of the form $\{z\in D: |f(z)| > R'\}$, for $R'>R$, must also have an unbounded boundary component. Thus we can choose an unbounded curve $l_0$, say, in $\log D$ on which $\operatorname{Re} F(w)=\log r$ and $l_0 \cap \{w:\operatorname{Re}w<\log r\} \neq \emptyset$. Then, by \Cref{expansion} part \ref{expansionb}, $F(l_0)=l$.
Let $\zeta_0 \in l_0$ with $F(\zeta_0) \in l$ and let $\zeta_u=\zeta_0+2\pi i u$, for $u \in \mathbb{Z}$. Finally, let $H_u$ be the branch of $F^{-1}$ that maps $F(\zeta_0)$ to $\zeta_u$, as given by \Cref{expansion} part \ref{expansionb}.

Since there exist no singular values of $f$ restricted to $D$ that lie in $A({r}/{\lambda}, \lambda r)$, we can obtain the following estimate on the derivative of $H_u$ by applying \Cref{expansion} part \ref{expansionc} for $w\in \{w:\log (r/\sqrt{\lambda})<\operatorname{Re} w < \log \sqrt{\lambda} r\}$:
\begin{equation}\label{Hu estimate}
\max_{w \in S_{r,s}}|(H_u)'(w)|<\frac{4 \pi}{\log(\lambda r) - \log(\sqrt{\lambda} r)}=\frac{4\pi}{\log \sqrt{\lambda}}.
\end{equation}


\textbf{Step 3.} Now, we estimate the diameter of $Q_{u,s}=H_u \circ G_s(Q)$, a second preimage of $Q$.  Recall that $G_s(Q) \subset S_{r,s}$, so
by \eqref{Gs estimate}~and~\eqref{Hu estimate},
\[\max_{w \in Q} |(H_u\circ G_s)'(w)|<\frac{4 \pi K}{a(r,v) \log \sqrt{\lambda}}.\]
So, there exists $d \in (0, 1)$ such that
\begin{equation}\label{DiameterD}\operatorname{diam} Q_{u,s} < \max_{w \in Q} |(H_u\circ G_s)'(w)| \operatorname{diam} Q < \frac{4 \pi \sqrt{2} K a(r,v)^{1-\tau} }{ 4 a(r,v) \log \sqrt{\lambda}} < d,\end{equation}
for $r$ sufficiently large.

Let \[Q'=\left\{w:\left|\operatorname{Re} w - \log |f(z_r)| \right|< \frac{1}{8} a(r,v)^{1-\tau}-d, \left|\operatorname{Im} w - \arg f(z_r) \right|< \frac{1}{8} a(r,v)^{1-\tau}-d\right\}\]
and
\[Q''=\left\{w:\left|\operatorname{Re} w - \log |f(z_r)| \right|< \frac{1}{16} a(r,v)^{1-\tau}, \left|\operatorname{Im} w - \arg f(z_r) \right|< \frac{1}{16} a(r,v)^{1-\tau}\right\},\]
where $d$ satisfies \eqref{DiameterD}.

\textbf{Step 4.} We now consider level curves of $F$ that meet $Q''$. Let $l_u$  be the image under $H_u$ of $l=\{w:\operatorname{Re} w = \log r\}$, that is a translate of the curve $l_0$ introduced in Step 2. Denote by $l_{u,Q'}$ the intersection of $l_u$ and $Q'$. From Step 2, we know that $l_u$ is unbounded, and so if $l_{u,Q'} \cap Q'' \neq \emptyset$, then 
\begin{equation}\label{lengthestimate}
\operatorname{length}(l_{u,Q'}) \geq  \frac{1}{16}a(r,v)^{1-\tau}-d.
\end{equation}
Note that since $l_0\cap l\neq \emptyset$ (from Step 2) there are at least $a(r,v)^{1-\tau}/16\pi$ values of $u$ such that $l_u \cap Q''\neq \emptyset$.

\textbf{Step 5.} Next, we obtain an estimate on $\sum_u \sum_s |(H_u \circ G_s)'(w)|$ for $w \in Q$, and $u$ and $s$ such that $H_u\circ G_s(Q)=Q_{u,s}\subset Q$.
Now, $l_{u,Q'} \subset H_u(l)$. Since $H_u$ is univalent on ${\{w:\log(r/\lambda)<\operatorname{Re}w<\log(\lambda r)\}}$, it follows from the Koebe distortion theorem (applied to a covering of $l$ by disks of a uniform size) that there exists a constant $C_2>0$ such that
\begin{equation}\label{lengthinequality}
\operatorname{length}(l_{u,Q'}) < C_2 \sum_s |(H_u)'(w_{r,s})|,
\end{equation}
for any $u$ such that $l_{u,Q'}\cap Q'' \neq \emptyset$, summing over all $s$ such that $H_u(w_{r,s}) \in l_{u,Q'}$. Note that if ${H_u(w_{r,s}) \in l_{u,Q'}}$, then $Q_{u,s} \subset Q$ by \eqref{DiameterD}.

The distortion of $H_u \circ G_s$ is uniformly bounded on $Q$, since $G_s:\hat{Q} \rightarrow S_{r,s}$ is univalent on $\hat{Q}\supset Q$, and $H_u$ is univalent  on $\{w:\log(r/\lambda)<\operatorname{Re}w<\log(\lambda r)\}$. Therefore, there exists a constant $C>0$ such that
\[|(H_u \circ G_s)'(w)|>\frac{1}{C}|(H_u \circ G_s)'(F(w_{r,s}))|, \quad \text{for}~ w \in Q.\]
So, by \eqref{Gwrs estimate}, \eqref{lengthestimate}, and \eqref{lengthinequality}, for $w \in Q$, and $u$ and $s$ such that $Q_{u,s} \subset Q$,
\begin{align*}
 \sum_s |(H_u \circ G_s)'(w)|&> \frac{1}{C}  \sum_s |(H_u \circ G_s)'(F(w_{r,s}))| \\
&= \frac{1}{C a(r,v)}\sum_s |(H_u)'(w_{r,s})| \\
&> \frac{\frac{1}{16}a(r,v)^{1-\tau}-d}{C C_2 a(r,v)}\\
&>\frac{a(r,v)^{(1-\tau)}}{C_3 a(r,v) },
\end{align*}
for $r$ sufficiently large, where $C_3>0$ is some constant. As we noted at the end of Step 4, there exist at least $a(r,v)^{1-\tau}/16\pi$ curves $l_u$ that meet $Q''$. Hence, for $w \in Q$, and $u$ and $s$ such that $Q_{u,s} \subset Q$,
\begin{equation}\label{PressureSum}
\sum_u \sum_s |(H_u \circ G_s)'(w)|>\frac{a(r,v)^{2(1-\tau)}}{C_3 16 \pi a(r,v) }.
\end{equation}

\textbf{Step 6.}
Following \cite[p. 622 - 623]{BKZ09}, we now obtain a conformal iterated function system from the family of maps $H_u \circ G_s: Q\rightarrow Q$. The sets  $H_u \circ G_s(Q)=Q_{u,s}$ are pairwise disjoint and the system gives a maximal compact invariant set, $J_B$, say. Note that by construction $F^2(J_B)=J_B$ and $F(J_B) \subset \cup_{s\in \mathbb{Z}} D_s$. Further, $\lim_{n\rightarrow \infty} |(F^n)'(w)| = \infty$ for every $w\in J_B$ by \eqref{Gs estimate}~and~\eqref{Hu estimate}.

Recall that the Hausdorff dimension of $J_B$ is the unique zero of the pressure function
\[P(t)=\lim_{n \rightarrow \infty} \frac{1}{n} \log \sum_{g^n} ||(g^n)'||^t,\]
where $g^n=g_{i_1} \circ \cdots \circ g_{i_n}$, with $g_{i_j}= H_u \circ G_s$ for $u$ and $s$ such that $Q_{u,s} \subset Q$. Further recall that the pressure function is strictly decreasing, so in order to prove that $\dim J_B>1$ it is sufficient to show that $P(1)>0$. For an introduction to the pressure function see \cite{Bowen} and \cite{KotUrb08}.

Now, estimating the pressure using \eqref{PressureSum},
\begin{align*}
P(1) &= \lim_{n \rightarrow \infty} \frac{1}{n} \log \sum_{g^n} ||(g^n)'|| \\
&\geq \lim_{n \rightarrow \infty} \frac{1}{n} \log \left(\inf_{w \in Q} \sum_{(u,s)} |(H_u \circ G_s)'(w)|\right)^n \\
&\geq \log \left(\frac{a(r,v)^{2(1-\tau)}}{C_3 16 \pi a(r,v) }\right)\\
&>0,
\end{align*}
for $r$ sufficiently large, since, by hypothesis, $\tau<1/2$. This implies that the Hausdorff dimension of the invariant set $J_B$ that arises from this system is greater than $1$.

Let $X = \exp(J_B \cup F(J_B))$. As $\exp \circ F=f \circ \exp$ and $J_B \cup F(J_B)$ is $F$-invariant, $X$ is $f$-invariant. Further, the exponential function is a smooth covering map, so $\dim X=\dim J_B \cup F(J_B)>1$. Finally, we have that $\lim_{n\rightarrow \infty} |(F^n)'(w)| = \infty$ for $w\in J_B$ by \eqref{Gs estimate} and \eqref{Hu estimate}. Therefore,  $\lim_{n\rightarrow \infty} |(f^n)'(z)| = \infty$ for $z\in X$. However, $f^n (X)$ is bounded, so by normality, $X \subset J(f)$. Therefore, $\dim J(f) \cap K(f)>1$.
\end{proof}

\section{Examples}\label{Examples}
In this section, we give two examples in order to illustrate the application of our results. First, we consider a function $f$ with one direct tract, which is not simply connected. However, this function satisfies the hypotheses of \Cref{Dim} and so the Hausdorff dimension of $J(f) \cap K(f)$ is strictly greater than one.
\begin{example}\label{example1}
Let 
\[f(z)= \cos(z)\exp(z).\]
Then
$\dim J(f) \cap K(f) >1$. 

Consider the direct tract $D$ of $f$ with boundary value $1$; see \Cref{cosexpfig}. Then $D$ contains $\{ iy: y \neq 0\}$. Note that the only zeros of $f$ are real and the only finite asymptotic value of $f$ is $0$, by the Denjoy-Carleman-Ahlfors theorem \cite[XI.4]{Nevanlinna}, since $f$ has order $1$ and is symmetric in the real axis. We show that the points where $f$ attains its maximum in the upper half-plane lie asymptotic to the line $y=x$ and that $D$ contains a disc satisfying \Cref{Dim}. 

\begin{figure}[!t]
  \centering
    \fbox{\includegraphics[width=0.4\textwidth]{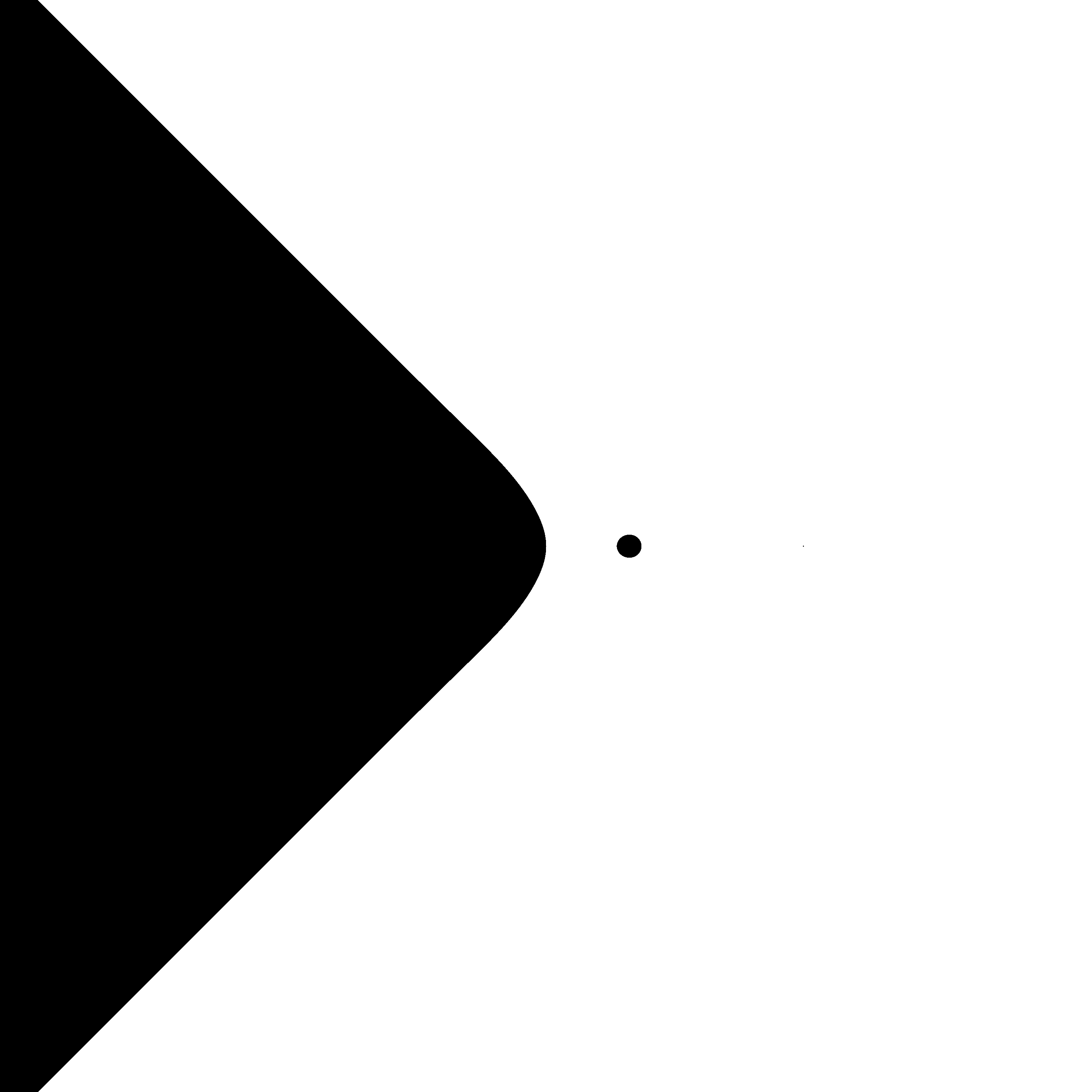}}
  \caption{The direct tract of $\cos(z)\exp(z)$ in white with boundary value $1$ and zeros at $\pi/2 + k\pi$ for $k\in \mathbb{Z}$. Its boundary includes the point $0$.}
  \label{cosexpfig}
\end{figure}

We use the following notation and results of  Tyler \cite{MaxCurves} on maximum modulus curves; that is curves on which $|f(z)|=M(|z|)$. First, let 
\[A(z) = \frac{zf'(z)}{f(z)} ~\text{and}~ B(z) = zA'(z).\] Further, let
\[ b(r)= \frac{d^2 \log M(r)}{d(\log r)^2}.\]
Then, by \cite[p. 2562]{MaxCurves} we have the following two properties:
\begin{enumerate}[(a)]
\item $|f(z)|  = M(r) \implies A(z)$ is real and positive, for $|z|=r$, \label{Conda}
\item $|f(z)|  = M(r) \implies B(x) = b(r) \geq 0$, for $x=r$.\label{Condb}
\end{enumerate}
For $f(z)= \cos(z)\exp(z)$, we have that $A(z) = z(1-\tan z)$. Now, $A(z)$ is real and positive either in various intervals on the real line or asymptotically close to the line $y=x$, since $\tan (z) \rightarrow i$ as $\operatorname{Im}z \rightarrow {\infty}$.

By considering $B(x)$, we show that there are no large maximum modulus points on the positive real axis. First,
\[B(x)= x(1-\tan x - x \sec^2 x) = x(1-\tan x - x(1+\tan^2 x)).\]
Hence, for $x>0$,
\[B(x) \geq 0 \iff x \leq \frac{1-\tan x}{1+\tan^2 x} \iff x \leq 0.43\ldots,\]
approximately. Therefore,  there are no maximum modulus points on the positive real axis greater than $0.44$.
So, the points satisfying both \ref{Conda} and \ref{Condb} in the upper half-plane must lie on an unbounded curve in the upper half-plane asymptotically close to the line $y=x$.

Now, all the zeros of $f$ lie on the real axis. Since, 
\[|f(x+iy)| = e^x (\cos^2 x+\sinh^2y)^\frac{1}{2} \geq e^x \sinh y>1, \quad \text{for} ~ x>0 ~\text{and} ~ |y| \geq 1,\]
 the direct tract $D$ contains $\{x+iy: x>0, |y| \geq 1\}$.
 
Hence for $\tau>0$ there exists a disc of radius $r^{1-\tau}$ with center on the line $y=x$ inside a tract of $f$. Further, the critical values of $f$ grow like the exponential function; that is, the modulus of the critical values of $f$ are $\exp(\pi k+\pi/4)/\sqrt{2} $ for $k \in \mathbb{Z}$. So, there exists $\lambda$ and an annulus $A(r/\lambda, \lambda r)$ which contains no singular values of the restriction of $f$ to $D$ for arbitrarily large $r$. Hence the conditions of \Cref{Dim} are satisfied, so we deduce from \Cref{Dim} that $\dim J(f) \cap K(f) >1$.

\end{example}

Next, we give an example with  a direct tract with no logarithmic singularities.   This example is the reciprocal of the entire function studied in \cite{BE08}; for an illustration of the tracts of this function, see \cite[Figure 1]{BE08}. 

\begin{example}\label{example2}
Consider the entire function
\[f(z)=\exp(-g(z)), ~\text{where} ~g(z)=\sum_{k=1}^{\infty} \left(\frac{z}{2^k}\right)^{2^k}.\]
Then $\dim J(f) \cap K(f) >1$. 

We will use several key results about $f$ from \cite[p. 254-258]{BE08}. The first is the existence of an infinite tree on which $f$ is very large. 
Introducing some notation, fix $0<\varepsilon \leq 1/8$ and set $r_n=(1+\varepsilon)2^{n+1}$ and $ r_n' = (1-2\varepsilon) 2^{n+2}$ for $n \in \mathbb{N}$. Then, for $j \in \{0, 1, \ldots, 2^n-1\}$, let
\[B_{j,n}=\left\{ r \exp \left(\frac{\pi i}{2^n}+\frac{2\pi ij}{2^n}\right): r_n \leq r \leq r_n' \right \}\]
and
\[C_{j,n}^{\pm}=\left\{ r \exp \left(\frac{\pi i}{2^n}+\frac{2\pi ij}{2^n} \pm \frac{r-r'_n}{r_{n+1}-r_n'}\frac{\pi i}{2^{n+1}}\right): r'_n \leq r \leq r_{n+1} \right \}.\]
We then obtain an infinite binary tree
\[T= [-i r_1, ir_1] \cup \bigcup_{n=1}^\infty \bigcup_{j=0}^{2^n-1} (B_{j,n} \cup C_{j,n}^{\pm}).\]
If $n$ is large enough, then $\operatorname{Re} g(z) < -2^{2^n}$, for $z\in B_{j,n} \cup C_{j,n}^{\pm}$, $j = 0, 1, \ldots, 2^n-1$.

The second key result in \cite{BE08} is that if $r_n \leq |z| \leq r_n'$, then
\[g(z)=(1+\eta(z))\left(\frac{z}{2^n}\right)^{2^n},\]
where $\eta(z) \rightarrow 0$ as $n\rightarrow \infty$. 

Let $\rho_n=(1+2\varepsilon)2^{n+1}$ and $\rho_n'=(1-3\varepsilon)2^{n+2}$.  
The authors then show a third key result that 
\[\left|zg'(z)-2^ng(z)\right|<1/2 |g(z)|,\]
for $\rho_n \leq |z| \leq \rho_n'$.

The final key result we need is that for $\rho_n \leq r \leq \rho_n'$, we have that $\arg g(re^{i \theta})$ is an increasing function of $\theta$ and increases by $2^n 2 \pi$ as $\theta$ increases by $2 \pi$. The authors conclude that for $\rho_n \leq r \leq \rho_n'$ the circle $\{z: |z| = r\}$ contains exactly $2^n$ arcs where $\operatorname{Re} g(re^{i\theta}) < \log \rho$ for $r$ and $n$ sufficiently large that $\{z: |z| = r\}\cap D \neq \emptyset$ and $\log \rho> -2^{2^n}$.

Let $D$ be a direct tract of $f$ with boundary value $R>0$. As $f$ has no zeros, $D$ is simply connected. Further, from \cite[Example 3.3]{IDingtracts}, $f$ has a bounded set of asymptotic values, and we can assume they all lie in $D(0,R)$. Hence, it remains to check that the critical values of~$f$ are suitably well separated in order to satisfy condition \ref{Dim2} in \Cref{Dim}. We deduce this by using the key results above discussed in \cite{BE08}, as well as an application of Rouch\'e's theorem and properties of level curves.

First, from the third key result above $g'(z) \neq 0$ in $\overline{A(\rho_n,\rho_n')}$ for any $n$ and so any critical points of $f$ must lie in $\cup_{n\in \mathbb{N}} A(\rho'_{n-1},\rho_n)$. Now, since 
\[\left|zg'(z)-2^ng(z)\right|<1/2 |g(z)|,\]
for $\rho_n \leq |z| \leq \rho_n'$, by Rouch\'e's theorem $g(z)$ and $zg'(z)$ have the same number of zeros in $D(0,\rho_n)$. Further, $g(z)$ has $2^n$ zeros in $D(0,\rho_n)$ by the second key result and another application of Rouch\'e's theorem. Therefore, $f$ has exactly $2^{n}-1$ critical points in $D(0,\rho_n)$ and hence  $2^{n-1}$ critical points in $A(\rho'_{n-1},\rho_n)$, for $n\geq 1$. 
All such critical points of $f$, apart from $0$, must lie in components of $A(\rho'_{n-1},\rho_n)\cap D$, for some $n\geq 1$. From each such critical point there must originate at least $4$ unbounded level curves which go to $\infty$ through $A(\rho_n,\rho'_n)\cap D$. By the fourth key result above, exactly $2$ of these curves can pass through each component of $A(\rho_n,\rho'_n)\cap D$, and the unions of these level curves must meet the tree $T$. Hence, the modulus of the level curves must be at least the minimum modulus of $f$ on the tree $T$ in $ A(\rho'_{n-1},\rho_n)\cap D$. Therefore, the modulus of each of the critical values of the critical points in $A(\rho'_{n-1},\rho_n)\cap D$ is at least $\min \{ |f(z)|:z \in T, |z| = \rho'_{n-1}\}\geq \exp(2^{2^{n-1}})$.

Now, label the moduli of the critical values arising from critical points of $f$ in $D$ as
\[a_{j,n}=|f(\zeta_{j,n})|,\]
where $\zeta_{j,n}\in A(\rho'_{n-1},\rho_n)\cap D$, $j=0,1,\ldots 2^{n-1}-1$, and $n\geq 1$. So,
\[a_{j,n}\geq \exp(2^{2^{n-1}}), \quad \text{for}~ j=0,1,\ldots 2^{n-1}-1, ~\text{and}~ n\geq 1.\]
Then we claim that for any $k>1$ there exist arbitrarily large $r$ such that $A(r,kr)$ contains no $a_{j,n}$, for $ j=0,1,\ldots 2^{n-1}-1$, and  $n\geq 1$. For otherwise there would exist a subsequence ${a_{j_m,n_m} \in A(k^{m-1},k^m)}$, for all $m \geq M$, say. Hence,
\[\exp(2^{2^{m-1}}) \leq \exp(2^{2^{n_m-1}})\leq a_{j_m,n_m}\leq k^m, \] 
for all $m\geq M$, which is impossible.
Therefore, we may apply \Cref{Dim} and obtain that $\dim J(f) \cap K(f) >1$.

\end{example}

\begin{remark}
Bergweiler and Karpi\'{n}ska \cite[Theorem 1.1]{BK10} show that if $f$ is a transcendental entire function that satisfies a certain regularity condition, then $\dim J(f) \cap I(f)=2$. This result applies to our \Cref{example1}, but not to \Cref{example2} since the regularity condition in \cite{BK10}  implies that $f$ has finite order \cite[p. 533]{BK10}, which is not the case for \Cref{example2} by a theorem of P\'olya \cite[Theorem 2.9]{HaymanMero}.
\end{remark}
\subsection*{Acknowledgements}
The author would like to thank his supervisors Gwyneth Stallard and Phil Rippon for their guidance and help in the preparation of this paper. The author also thanks Walter Bergweiler for a helpful conversation about direct tracts.
\bibliographystyle{abbrv}
\bibliography{mybib}{}

\begin{thebibliography}{10}

\bibitem{Baker75}
I.~N. Baker.
\newblock The domains of normality of an entire function.
\newblock {\em Ann. Acad. Sci. Fenn. Ser. A I Math.}, 1(2):277--283, 1975.

\bibitem{FiniteOrderBaranski}
K.~Bara\'{n}ski.
\newblock Hausdorff dimension of hairs and ends for entire maps of finite
  order.
\newblock {\em Math. Proc. Cambridge Philos. Soc.}, 145(3):719--737, 2008.

\bibitem{BKZ09}
K.~Bara\'nski, B.~Karpi\'nska, and A.~Zdunik.
\newblock Hyperbolic dimension of {J}ulia sets of meromorphic maps with
  logarithmic tracts.
\newblock {\em Int. Math. Res. Not. IMRN}, (4):615--624, 2009.

\bibitem{B93}
W.~Bergweiler.
\newblock Iteration of meromorphic functions.
\newblock {\em Bull. Amer. Math. Soc. (N.S.)}, 29(2):151--188, 1993.

\bibitem{SizeOfTracts}
W.~Bergweiler.
\newblock The size of {W}iman-{V}aliron discs.
\newblock {\em Complex Var. Elliptic Equ.}, 56(1-4):13--33, 2011.

\bibitem{BergweilerBounded}
W.~Bergweiler.
\newblock On the set where the iterates of an entire function are bounded.
\newblock {\em Proc. Amer. Math. Soc.}, 140(3):847--853, 2012.

\bibitem{BE08}
W.~Bergweiler and A.~Eremenko.
\newblock Direct singularities and completely invariant domains of entire
  functions.
\newblock {\em Illinois J. Math.}, 52(1):243--259, 2008.

\bibitem{BK10}
W.~Bergweiler and B.~Karpi\'{n}ska.
\newblock On the {H}ausdorff dimension of the {J}ulia set of a regularly
  growing entire function.
\newblock {\em Math. Proc. Cambridge Philos. Soc.}, 148(3):531--551, 2010.

\bibitem{Tracts}
W.~Bergweiler, P.~J. Rippon, and G.~M. Stallard.
\newblock Dynamics of meromorphic functions with direct or logarithmic
  singularities.
\newblock {\em Proc. Lond. Math. Soc. (3)}, 97(2):368--400, 2008.

\bibitem{BishopDim1}
C.~J. Bishop.
\newblock A transcendental {J}ulia set of dimension 1.
\newblock {\em Invent. Math.}, 212(2):407--460, 2018.

\bibitem{Bowen}
R.~Bowen.
\newblock Hausdorff dimension of quasicircles.
\newblock {\em Inst. Hautes \'{E}tudes Sci. Publ. Math.}, (50):11--25, 1979.

\bibitem{EL92}
A.~E. Eremenko and M.~Y. Lyubich.
\newblock Dynamical properties of some classes of entire functions.
\newblock {\em Ann. Inst. Fourier (Grenoble)}, 42(4):989--1020, 1992.

\bibitem{HaymanMero}
W.~K. Hayman.
\newblock {\em Meromorphic functions}.
\newblock Oxford Mathematical Monographs. Clarendon Press, Oxford, 1964.

\bibitem{Hayman74}
W.~K. Hayman.
\newblock The local growth of power series: a survey of the {W}iman-{V}aliron
  method.
\newblock {\em Canad. Math. Bull.}, 17(3):317--358, 1974.

\bibitem{Iversen}
F.~Iversen.
\newblock {\em Recherches sur les fonctions inverses des fonctions
  m{\'e}romophes}.
\newblock PhD thesis, Helsingfors, 1914.

\bibitem{KotUrb08}
J.~Kotus and M.~Urba\'{n}ski.
\newblock Fractal measures and ergodic theory of transcendental meromorphic
  functions.
\newblock In {\em Transcendental dynamics and complex analysis}, volume 348 of
  {\em London Math. Soc. Lecture Note Ser.}, pages 251--316. Cambridge Univ.
  Press, Cambridge, 2008.

\bibitem{M38}
A.~J. Macintyre.
\newblock Wiman's method and the `flat regions' of integral functions.
\newblock {\em Quart. J. Math., Oxford Ser. 9}, pages 81--88, 1938.

\bibitem{McMullen87}
C.~McMullen.
\newblock Area and {H}ausdorff dimension of {J}ulia sets of entire functions.
\newblock {\em Trans. Amer. Math. Soc.}, 300(1):329--342, 1987.

\bibitem{M81}
M.~Misiurewicz.
\newblock On iterates of {$e^{z}$}.
\newblock {\em Ergodic Theory Dynam. Systems}, 1(1):103--106, 1981.

\bibitem{Nevanlinna}
R.~Nevanlinna.
\newblock {\em Analytic functions}.
\newblock Translated from the second German edition by Phillip Emig. Die
  Grundlehren der mathematischen Wissenschaften, Band 162. Springer-Verlag, New
  York-Berlin, 1970.

\bibitem{OsborneBounded}
J.~Osborne.
\newblock Connectedness properties of the set where the iterates of an entire
  function are bounded.
\newblock {\em Math. Proc. Cambridge Philos. Soc.}, 155(3):391--410, 2013.

\bibitem{Ransford}
T.~Ransford.
\newblock {\em Potential theory in the complex plane}, volume~28 of {\em London
  Mathematical Society Student Texts}.
\newblock Cambridge University Press, Cambridge, 1995.

\bibitem{Bakerdomains}
P.~J. Rippon and G.~M. Stallard.
\newblock Singularities of meromorphic functions with {B}aker domains.
\newblock {\em Math. Proc. Cambridge Philos. Soc.}, 141(2):371--382, 2006.

\bibitem{RRRS}
G.~Rottenfusser, J.~R\"{u}ckert, L.~Rempe, and D.~Schleicher.
\newblock Dynamic rays of bounded-type entire functions.
\newblock {\em Ann. of Math. (2)}, 173(1):77--125, 2011.

\bibitem{FiniteOrderSchubert}
H.~Schubert.
\newblock {\em {\"U}ber die Hausdorff-Dimension der Juliamenge von Funktionen
  endlicher Ordnung}.
\newblock PhD thesis, 2007.

\bibitem{StallardDimensionsB}
G.~M. Stallard.
\newblock The {H}ausdorff dimension of {J}ulia sets of entire functions. {II}.
\newblock {\em Math. Proc. Cambridge Philos. Soc.}, 119(3):513--536, 1996.

\bibitem{StallardDimensions}
G.~M. Stallard.
\newblock The {H}ausdorff dimension of {J}ulia sets of entire functions. {IV}.
\newblock {\em J. London Math. Soc. (2)}, 61(2):471--488, 2000.

\bibitem{Tsuji}
M.~Tsuji.
\newblock {\em Potential theory in modern function theory}.
\newblock Chelsea Publishing Co., New York, 1975.
\newblock Reprinting of the 1959 original.

\bibitem{MaxCurves}
T.~F. Tyler.
\newblock Maximum curves and isolated points of entire functions.
\newblock {\em Proc. Amer. Math. Soc.}, 128(9):2561--2568, 2000.

\bibitem{valiron49}
G.~Valiron.
\newblock {\em Lectures on the general theory of integral functions}.
\newblock \'Edouard Privat, Toulouse, 1923.

\bibitem{IDingtracts}
J.~Waterman.
\newblock Identifying logarithmic tracts.
\newblock To appear in \textit{Ann. Acad. Sci. Fenn.} Preprint
  \href{https://arxiv.org/abs/1902.04330}{arXiv:1902.04330}.

\end{thebibliography}
\end{document}